\documentclass[12pt]{amsart}
\usepackage{amsmath,amssymb,amsfonts,amsthm}
\usepackage[all,cmtip]{xy}
\usepackage{fullpage}
\usepackage[OT2]{fontenc}
\usepackage[T5,T1]{fontenc}
\usepackage{hyperref}
\usepackage[francais,english]{babel}
\begin{document}
\theoremstyle{plain}
\newtheorem{thm}{Theorem}
\newtheorem{lem}[thm]{Lemma}
\newtheorem{cor}[thm]{Corollary}
\newtheorem{prop}[thm]{Proposition}
\newtheorem{rem}[thm]{Remark}
\newtheorem{defn}[thm]{Definition}
\newtheorem{ex}[thm]{Example}
\title[Index theorem]{Algebraic Atiyah-Singer index theorem}
\author{{Nguyen Le Dang Thi}}
\email{nguyen.le.dang.thi@gmail.com}
\date{06. 02. 2017}          
\begin{abstract}
The aim of this work is to give an algebraic weak version of the Atiyah-Singer index theorem. We compute then a few small examples with the elliptic differential operator of order $\leq 1$ coming from the Atiyah class in $\mathrm{Ext}^1_{\mathcal{O}_X}(\mathcal{O}_X,\Omega^1_{X/k})$, where $X \longrightarrow \mathrm{Spec}(k)$ is a smooth projective scheme over a perfect field $k$.
\end{abstract}
\maketitle
We follow Grothendieck \cite[\S 16.8]{EGA4} to recall briefly the notion of differential operators. Let $f : X \longrightarrow S$ be a morphism of schemes. Consider the Cartesian square
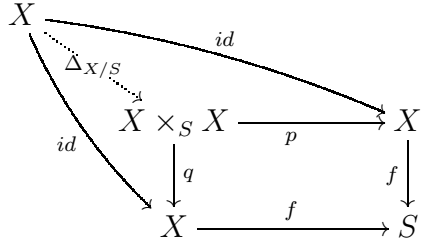
\[\xymatrix{ X \ar@/_/[ddr]_{id} \ar@/^/[drrr]^{id} \ar@{.>}[dr]|-{\Delta_{X/S}}\\
&X \times_S X \ar[d]^q \ar[rr]_p && X\ar[d]_f \\
&X \ar[rr]^f && S}\]
The diagonal $\Delta_{X/S}: X \longrightarrow X \times_S X$ is an immersion. One defines the $n$-th normal invariant of $\Delta_{X/S}$ as 
\[\mathcal{P}^n_{X/S} = \Delta_{X/S}^{-1}(\mathcal{O}_{X \times_S X})/I^{n+1}_{f}.\]
It is clear that $\{\mathcal{P}^n_{X/S}\}_n$ form a projective system. One defines 
\[\mathcal{P}^{\infty}_{X/S} = \varprojlim_n \mathcal{P}^n_{X/S}.\]
The first projection $p : X \times_S X \longrightarrow X$ induces an $\mathcal{O}_X$-algebra structure on $\mathcal{P}^n_{X/S}$ and the second projection $q : X \times_S X \rightarrow X$ induces a morphism 
\[d^n_{X/S}: \mathcal{O}_X \longrightarrow \mathcal{P}^n_{X/S}.\]
If $\mathcal{F } \in \mathcal{O}_X-Mod$ is an $\mathcal{O}_X$-module, one defines 
\[\mathcal{P}^n_{X/S}(\mathcal{F}) = \mathcal{P}^n_{X/S} \otimes_{\mathcal{O}_X} \mathcal{F}.\]
Let $\mathcal{F}, \mathcal{G} \in \mathcal{O}_X-Mod$ be two $\mathcal{O}_X$-modules. Let $D : \mathcal{F} \longrightarrow \mathcal{G}$ be a morphism of the underlying abelian sheaves. $D$ is called a differential operator of order $\leq n$ relative $S$, if there exists a morphism of $\mathcal{O}_X$-modules $u : \mathcal{P}^n_{X/S}(\mathcal{F}) \rightarrow \mathcal{G}$, such that the following diagram commutes:
\[\xymatrix{\mathcal{F} \ar[rr]^D \ar[d]_{d^n_{X/S}} && \mathcal{G} \\ \mathcal{P}^n_{X/S}(\mathcal{F}) \ar[rru]_u  }\]  
The morphism $D : \mathcal{F} \rightarrow \mathcal{G}$ between the underlying abelian sheaves is not $\mathcal{O}_X$-linear, but it is $f^{-1}\mathcal{O}_S$-linear. Let us denote by $\mathrm{Diff}^n_{X/S}(\mathcal{F},\mathcal{G})$ the set of all differential operator of order $\leq n$ between $\mathcal{F}$ and $\mathcal{G}$. One defines
\[\mathrm{Diff}_{X/S}(\mathcal{F},\mathcal{G}) = \bigcup_{n} \mathrm{Diff}^n_{X/S}(\mathcal{F},\mathcal{G}).\]
It follows from the definition \cite[16.8.3.1]{EGA4} that one has an isomorphism of abelian groups
\[\mathrm{Hom}_{\mathcal{O}_X}(\mathcal{P}^n_{X/S}(\mathcal{F}),\mathcal{G}) \stackrel{\cong}{\longrightarrow} \mathrm{Diff}^n_{X/S}(\mathcal{F},\mathcal{G}).\]
If $f : X \rightarrow S$ is locally of finite type, then $\mathcal{P}^n_{X/S}$ is a quasi-coherent $\mathcal{O}_X$-module of finite type (cf. \cite[Prop. 16.3.9]{EGA4}). If $f : X \rightarrow S$ is locally of finite presentation, then $\mathcal{P}^n_{X/S}$ is a quasi-coherent $\mathcal{O}_X$-module of finite presentation (cf. \cite[Cor. 16.4.22]{EGA4}). Consequently, if $\mathcal{F}$ is a quasi-coherent $\mathcal{O}_X$-module of finite type resp. locally of finite presentation and $f : X \rightarrow S$ is locally of finite type resp. locally of finite presentation, then $\mathcal{P}^n_{X/S}(\mathcal{F}) = \mathcal{P}^n_{X/S} \otimes_{\mathcal{O}_X}\mathcal{F}$ is also a quasi-coherent $\mathcal{O}_X$-module of finite type resp. locally of finite presentation. If $S$ is locally Noetherian and $f : X \rightarrow S$ is proper, then by \cite[Thm. 3.2.1]{EGA3} the higher direct image sheaf 
\[R^qf_*(\mathcal{F}) = a_{Zar}(U \mapsto H^q_{Zar}(f^{-1}U,\mathcal{F}))\] 
is coherent for any $\mathcal{F} \in Coh(X)$ and $\forall q \geq 0$. Consequently, if $X \longrightarrow \mathrm{Spec}(A)$ is proper over a Noetherian ring $A$, then for all coherent sheaf $\mathcal{F} \in Coh(X)$ the Zariski cohomology groups $H^q_{Zar}(X,\mathcal{F})$ are $A$-modules of finite type $\forall q \geq 0$. A differential operator $D : \mathcal{F} \longrightarrow \mathcal{G}$ induces a homomorphism of abelian groups 
\[\Gamma_{ab}(X,\mathcal{F})  \stackrel{defn}{=} \mathrm{Hom}_{\underline{Ab}}(\mathbb{Z},\mathcal{F}) \longrightarrow \mathrm{Hom}_{\underline{Ab}}(\mathbb{Z},\mathcal{G}) \stackrel{defn}{=} \Gamma_{ab}(\mathbb{Z},\mathcal{G}),\] 
where we write $\underline{Ab}$ for the category of abelian sheaves on $X$ and $\mathbb{Z}$ is the terminal object in $\underline{Ab}$. In general, we can not say anything about the kernel or cokernel of this homomorphism. However, $D$ induces also a homomorphism of $\Gamma_{mod}(X,\mathcal{O}_X)$-modules 
\begin{multline*} 
\Gamma_{mod}(X,\mathcal{P}^n_{X/S} \otimes_{\mathcal{O}_X} \mathcal{F}) = \mathrm{Hom}_{\mathcal{O}_X-Mod}(\mathcal{O}_X,\mathcal{P}^n_{X/S} \otimes_{\mathcal{O}_X} \mathcal{F}) \longrightarrow \mathrm{Hom}_{\mathcal{O}_X-Mod}(\mathcal{O}_X,\mathcal{G}) \\ = \Gamma_{mod}(X,\mathcal{G}).
\end{multline*}
This motivates us to give the following definition: 
\begin{defn}\begin{enumerate}
\item Let $X \longrightarrow \mathrm{Spec}(k)$ be a $k$-scheme of finite type, where $k$ is a field and $D : \mathcal{F} \rightarrow \mathcal{G}$ be a differential operator between two $\mathcal{O}_X$-modules of order $\leq n$. $D$ is called a Fredholm operator, if the $k$-linear morphism between $k$-vector spaces
\[D(X): \Gamma_{mod}(X,\mathcal{P}^n_{X/k} \otimes_{\mathcal{O}_X}\mathcal{F}) \longrightarrow \Gamma_{mod}(X,\mathcal{G})\] 
has finite-dimensional kernel and cokernel. 
\item Let $X \rightarrow \mathrm{Spec}(A)$ be a scheme of finite type  over a Noetherian ring $A$ and $D : \mathcal{F} \rightarrow \mathcal{G}$ be a differential operator between two $\mathcal{O}_X$-modules. $D$ is called a Fredholm operator, if the kernel and cokernel of the $A$-module morphism 
\[D(X): \Gamma_{mod}(X, \mathcal{P}^n_{X/A} \otimes_{\mathcal{O}_X}\mathcal{F}) \longrightarrow \Gamma_{mod}(X,\mathcal{G})\]
are $A$-modules of finite type. 
\end{enumerate}
\end{defn} 
For a proper $A$-scheme $X$, where $A$ is a Noetherian ring, any differential operator $D : \mathcal{F} \longrightarrow \mathcal{G}$ between coherent $\mathcal{O}_X$-modules is a Fredholm operator. 
\begin{defn}
Let $X$ be a proper $k$-scheme and $D : \mathcal{F} \longrightarrow \mathcal{G}$ be a differential operator between coherent $\mathcal{O}_X$-modules. We define the index of $D$ to be 
\[Ind(D) \stackrel{defn}{=} dim_k(Ker(D(X)) - dim_k(Coker(D(X)).\] 
\end{defn} 
Let $X$ be any scheme and $QsCoh(X)$ be the category of quasi-coherent $\mathcal{O}_X$-modules. Following Grothendieck \cite[\S 1.7]{EGA2} we define a vector bundle associated to a quasi-coherent sheaf $\mathcal{E} \in QsCoh(X)$ to be the $X$-scheme
\[\mathbb{V}(\mathcal{E}) \stackrel{defn}{=} \mathrm{Spec}(Sym_{\mathcal{O}_X}(\mathcal{E})).\]
The structure morphism $\pi: \mathbb{V}(\mathcal{E}) \rightarrow X$ is an affine morphism. For any morphism $f: Y \rightarrow X$, we have 
\begin{multline*}
\mathrm{Hom}_{Sch/X}(Y,\mathbb{V}(\mathcal{E})) \stackrel{\cong}{\longrightarrow} \mathrm{Hom}_{\mathcal{O}_X-Alg}(Sym_{\mathcal{O}_X}(\mathcal{E}),f_*\mathcal{O}_Y) \stackrel{\cong}{\longrightarrow} \mathrm{Hom}_{\mathcal{O}_X-Mod}(\mathcal{E},f_*\mathcal{O}_Y).
\end{multline*}
We call $s_0 : X \rightarrow \mathbb{V}(\mathcal{E})$ the $0$-section, if it corresponds to the $0$-homomorphism $\mathcal{E} \rightarrow \mathcal{O}_X$. Now we assume $\mathcal{E}$ is an $\mathcal{O}_X$-module of finite type. Then $\pi: \mathbb{V}(\mathcal{E}) \rightarrow X$ is a morphism of finite type \cite[Prop. 1.7.11 (ii)]{EGA2}. So if $X$ is a Noetherian scheme, then $\mathbb{V}(\mathcal{E})$ is also a Noetherian scheme. Consider the locally small abelian category $Coh(\mathbb{V}(\mathcal{E}))$ of coherent sheaves on $\mathbb{V}(\mathcal{E})$. We let 
\[K_0(\mathbb{V}(\mathcal{E})) \stackrel{defn}{=} K(Coh(\mathbb{V}(\mathcal{E})))\]
to be the Grothendieck group of $\mathbb{V}(\mathcal{E})$. By \cite[Prop. 1.7.15]{EGA2} the morphism $s_0: X \longrightarrow \mathbb{V}(\mathcal{E})$ is a closed immersion. Let $j: \mathbb{V}(\mathcal{E}) - s_0(X) \longrightarrow \mathbb{V}(\mathcal{E})$ be the complement open immersion. By \cite[\S IX, Prop. 1.1]{SGA6} one has a localization exact sequence
\begin{equation}\label{eq1}
K_0(X) \stackrel{s_{0*}}{\longrightarrow} K_0(\mathbb{V}(\mathcal{E})) \stackrel{j^*}{\longrightarrow} K_0(\mathbb{V}(\mathcal{E}) - s_0(X)) \longrightarrow 0.
\end{equation}
If $f : X \rightarrow S$ is the structure morphism of $X$ over a base scheme $S$, then any differential operator $D \in \mathrm{Diff}^n_{X/S}(\mathcal{F},\mathcal{G})$ of order $\leq n$ between $\mathcal{O}_X$-modules defines a two terms complex of the underlying abelian sheaves 
\[0 \longrightarrow \mathcal{F} \stackrel{D}{\longrightarrow} \mathcal{G} \longrightarrow 0.\]
Or equivalently, a two terms complex of $\mathcal{O}_X$-modules
\[0 \longrightarrow \mathcal{P}^n_{X/S} \otimes_{\mathcal{O}_X} \mathcal{F} \longrightarrow \mathcal{G} \longrightarrow 0.\]
We obtain then a two terms complex of $\mathcal{O}_{\mathbb{V}(\mathcal{E})}$-modules 
\[0 \longrightarrow \pi^*(\mathcal{P}^n_{X/S}\otimes_{\mathcal{O}_X} \mathcal{F}) \longrightarrow \pi^*\mathcal{G} \longrightarrow 0.\] 
Remark that if $S$ is Noetherian and $f : X \longrightarrow S$ is of finite type, then $X$ is also Noetherian, so $\pi^*(\mathcal{F})$, $\pi^*(\mathcal{P}^n_{X/S})$ and $\pi^*(\mathcal{G})$ are coherent $\mathcal{O}_{\mathbb{V}(\mathcal{E})}$-modules, if $\mathcal{F}$ and $\mathcal{G}$ are coherent. Moreover, by \cite[Prop. 16.4.5]{EGA4} there is a canonical isomorphism 
\[\pi^*(\mathcal{P}^n_{X/S}) \stackrel{\cong}{\longrightarrow} \mathcal{P}^n_{\mathbb{V}(\mathcal{E})/X}. \] 
\begin{defn}
Let $f : X \rightarrow S$ be a morphism of finite type, where $S$ is a Noetherian scheme. Let $\mathcal{E} \in \mathcal{O}_X-Mod$ be an $\mathcal{O}_X$-module of finite type and $\pi: \mathbb{V}(\mathcal{E}) \rightarrow X$ be the associated vector bundle. Let $D \in \mathrm{Diff}^n_{X/S}(\mathcal{F},\mathcal{G})$ be a differential operator between coherent $\mathcal{O}_X$-modules of order $\leq n$. Then $D$ is called an elliptic operator with respect to $\mathcal{E}$, if there is an isomorphism
\[\sigma(D): \mathcal{P}^n_{\mathbb{V}(\mathcal{E})/X} \otimes_{\mathcal{O}_{\mathbb{V}(\mathcal{E})}} \pi^*\mathcal{F}_{| \mathbb{V}(\mathcal{E})-s_0(X)} \stackrel{\cong}{\longrightarrow} \pi^*\mathcal{G}_{| \mathbb{V}(\mathcal{E})- s_0(X)},\]
The isomorphism $\sigma(D)$ is called the symbol of $D$.  
\end{defn}
By the localization exact sequence \ref{eq1}, an elliptic operator $D$ of order $\leq n$ with respect to an $\mathcal{O}_X$-module of finite type $\mathcal{E}$ defines a class $[\mathcal{P}^n_{\mathbb{V}(\mathcal{E})/X}]\cdot[\pi^*\mathcal{F}] - [\pi^*\mathcal{G}] \in K_0(\mathbb{V}(\mathcal{E}))$, which becomes $0$ after restricting to the complement of the $0$-section.
If $f : X \rightarrow S$ a morphism locally of finite type, then the sheaf of relative differentials $\Omega^1_{X/S} = I_f/I_f^2$ is an $\mathcal{O}_X$-module of finite type.  
\begin{defn}
Let $X \rightarrow S$ be a morphism of finite type over a Noetherian scheme $S$. A differential operator $D \in \mathrm{Diff}_{X/S}(\mathcal{F},\mathcal{G})$ between coherent $\mathcal{O}_X$-modules is called elliptic, if it is elliptic with respect to $\Omega^1_{X/S}$.  
\end{defn} 
Let $X$ be a Noetherian scheme. Denote by $Vect(X)$ the locally small exact category of locally free sheaves of finite rank on $X$. We let 
\[K^0(X) \stackrel{defn}{=} K(Vect(X)).\]  
For an arbitrary morphism of Noetherian schemes $f: Y \rightarrow X$ one has a functorial pullback homomorphism
\[f^* : K^0(X) \longrightarrow K^0(Y), \quad [\mathcal{E}] \mapsto [f^* \mathcal{E}].\]
For a proper morphism $f : Y \longrightarrow X$ between Noetherian schemes there is a pushforward homomorphism 
\[f_* : K_0(Y) \longrightarrow K_0(X), \quad \mathcal{F} \mapsto \sum_i (-1)^i[R^if_*\mathcal{F}],\]
which is well-defined as an exact sequence of coherent $\mathcal{O}_Y$-modules 
\[0 \longrightarrow \mathcal{F}' \longrightarrow \mathcal{F} \longrightarrow \mathcal{F}'' \longrightarrow 0 \]
gives rise to an exact sequence of coherent $\mathcal{O}_X$-modules 
\[\cdots \longrightarrow R^if_*(\mathcal{F}') \longrightarrow R^if_*(\mathcal{F}) \longrightarrow R^if_*(\mathcal{F}'') \longrightarrow R^{i+1}f_*(\mathcal{F}') \longrightarrow \cdots \]
The functoriality of the pushforward follows easily from the Leray-Grothendieck spectral sequence 
\[E^{p,q}_2 = R^pg_*(R^qf_*\mathcal{F}) \Rightarrow R^{p+q}(g \circ f)_*\mathcal{F}.\]
If $X$ is a regular Noetherian scheme, then the canonical homomorphism induced by the obvious exact embedding of categories 
\[K^0(X) \stackrel{\delta}{\longrightarrow} K_0(X)\]
is an isomorphism, since every coherent sheaf on a regular scheme has a finite locally free resolution.

Let now $S$ be an arbitrary scheme. We denote by $SH(S)$ the stable motivic homotopy category together with the formalism of six functors $(f^*,f_*,f_!,f^!,\wedge, \underline{\mathrm{Hom}})$ as in \cite{Ay08}, \cite{CD12} and \cite[Appendix C]{Hoy14}. Let $\mathbf{E}$ be a motivic ring spectrum parameterized by schemes. We will assume $\mathbf{E}$ is stable by base change, i.e. for any morphism of schemes $f: X \longrightarrow Y$ there is an isomorphism of spectra in $SH(X)$
\[f^*\mathbf{E}_Y \stackrel{\cong}{\longrightarrow} \mathbf{E}_X.\]
Let $S$ be a scheme and $\mathbf{E}_S \in SH(S)$ be a motivic ring spectrum. The unit $\varphi_S : \mathbb{S}^0 \rightarrow \mathbf{E}_S$ gives rise to a class
\[[\varphi_S] \in \widetilde{\mathbf{E}}^{2,1}(\mathbb{P}^1_S) \cong \mathbf{E}^{0,0}(\mathbb{P}^1_S).\]
We have a tower of $S$-schemes given by the obvious embeddings 
\[\mathbb{P}^1_S \longrightarrow \mathbb{P}^2_S \longrightarrow \cdots \longrightarrow \mathbb{P}^n_S \longrightarrow \cdots \] 
Let $\mathbb{P}^{\infty}_S = colim_n \mathbb{P}^n_S$, which is an object in the pointed unstable motivic homotopy category of Morel-Voevodsky $Ho_{\mathbb{A}^1,\bullet}(S)$ \cite{MV99}. Let $i : \mathbb{P}^1_S \longrightarrow \mathbb{P}^{\infty}_S$ be the obvious map in $Ho_{\mathbb{A}^1,\bullet}(S)$, which gives rise to a map $\Sigma^{\infty}(i): \Sigma^{\infty}(\mathbb{P}^1_S)_+ \longrightarrow \Sigma^{\infty}(\mathbb{P}^{\infty}_S)_+$. $\mathbf{E}_S$ is called oriented, if there is a class 
\[\widetilde{\mathbf{E}}^{2,1}(\mathbb{P}^{\infty}_S) \ni c_S : \Sigma^{\infty}(\mathbb{P}^{\infty}_S)_+) \longrightarrow \mathbf{E}_S \wedge S^{2,1},\]
such that $i^*(c_S) = [\varphi_S]$. We say $\mathbf{E}$ is oriented, if $\mathbf{E}_S$ is oriented for any scheme $S$ and for any morphism $f : T \rightarrow S$ one has $f^*(c_S ) = c_T$. If $S$ is regular, by \cite[\S 4, 1.15, 3.7]{MV99} there is a canonical isomorphism $B\mathbb{G}_m \cong \mathbb{P}^{\infty}_S$ in $Ho_{\mathbb{A}^1,\bullet}(S)$, which gives to a canonical isomorphism
\[Pic(S) \stackrel{\cong}{\longrightarrow} [S_+,B\mathbb{G}_m] \stackrel{\cong}{\longrightarrow} [S_+,\mathbb{P}^{\infty}_S].\]
For an oriented motivic ring spectrum $\mathbf{E}$, one can define the first Chern class as 
\begin{multline*}
c_1 : Pic(S) \stackrel{\cong}{\longrightarrow} [S_+,\mathbb{P}^{\infty}_S] \stackrel{\Sigma^{\infty}}{\longrightarrow} \mathrm{Hom}_{SH(S)}(\Sigma^{\infty}S_+,\Sigma^{\infty}\mathbb{P}^{\infty}_S) \stackrel{(c_S)_*}{\longrightarrow} \mathrm{Hom}_{SH(S)}(\Sigma^{\infty}S_+,\mathbf{E}_S \wedge S^{2,1}) \\ \stackrel{=}{\longrightarrow} \mathbf{E}^{2,1}(S).
\end{multline*}
Assume $S$ is a regular scheme. Let $X \longrightarrow S$ be a smooth $S$-scheme. For a vector bundle $\mathbb{V}(\mathcal{E})$ associated to a locally free $\mathcal{O}_X$-module $\mathcal{E}$ of finite rank $r$, there is an isomorphism (cf. \cite[Thm. 2.11]{NSO09}) 
\[\bigoplus_{i=0}^{r-1}\mathbf{E}^{*-2i,*-i}(X) \stackrel{\cong}{\longrightarrow} \mathbf{E}^{*,*}(\mathbb{P}(\mathcal{E})), \quad (x_0,\cdots,x_{r-1}) \mapsto \sum_{i=0}^{r-1}p^*(x_i)\cup c_1^i(\mathcal{O}_{\mathbb{P}(\mathcal{E})}(-1)),\]
where $\mathbb{P}(\mathcal{E}) = \mathrm{Proj}(Sym_{\mathcal{O}_X}(\mathcal{E})) \stackrel{p}{\longrightarrow} X$ is the projective bundle associated to $\mathcal{E}$ \cite[\S 4]{EGA2}. So $\mathbf{E}^{*,*}(\mathbb{P}(\mathcal{E}))$ is a free module over $\mathbf{E}^{*,*}(X)$ with the basis 
\[\{1,c_1(\mathcal{O}_{\mathbb{P}(\mathcal{E})}(-1)),\cdots,c_1(\mathcal{O}_{\mathbb{P}(\mathcal{E})}(-1))^{r-1}) \}.\]
So for a vector bundle $\mathbb{V}(\mathcal{E})$ of rank $r$ one can define the higher Chern classes 
\[\mathbf{E}^{2i,i}(X) \ni c_i(\mathbb{V}(\mathcal{E})), \quad \sum_{i=0}^r p^*(c_i(\mathbb{V}(\mathcal{E})) \cup (-c_1(\mathcal{O}_{\mathbb{P}(\mathcal{E})}(-1))^{r-i} = 0,\]
where one puts $c_0(\mathbb{V}(\mathcal{E})) = 1$ and $c_i(\mathbb{V}(\mathcal{E})) = 0$ for $i \notin [0,r]$. Recall that one has a homotopy cofiber sequence \cite[\S 3]{MV99}  
\[\mathbb{P}(\mathcal{E}) \stackrel{i}{\longrightarrow} \mathbb{P}(\mathcal{E} \oplus \mathcal{O}_X) \stackrel{q}{\longrightarrow} Th_X(\mathbb{V}(\mathcal{E})),\]
where $Th_X(\mathbb{V}(\mathcal{E}))$ denotes the Thom space of the vector bundle $\mathbb{V}(\mathcal{E})$. The homotopy cofiber sequence induces a long exact sequence
\[\cdots \longrightarrow \mathbf{E}^{*,*}(Th_X(\mathbb{V}(\mathcal{E}))) \stackrel{q^*}{\longrightarrow} \mathbf{E}^{*,*}(\mathbb{P}(\mathcal{E} \oplus \mathcal{O}_X)) \stackrel{i^*}{\longrightarrow} \mathbf{E}^{*,*}(\mathbb{P}(\mathcal{E})) \longrightarrow \cdots \]
The projective bundle formula tells us that $i^*$ is a split epimorphism, so $\mathbf{E}^{*,*}(Th_X(\mathbb{V}(\mathcal{E})))$ is a free module of rank $1$ over $\mathbf{E}^{*,*}(X)$, which is just $Ker(i^*)$. The Thom class is the unique class in $\mathbf{E}^{2r,r}(Th_X(\mathbb{V}(\mathcal{E})))$ 
\[t_{\mathbb{V}(\mathcal{E})} = (q^*)^{-1}(\sum_{i=0}^r p^*(c_i(\mathbb{V}(\mathcal{E})) \cup (-c_1(\mathcal{O}_{\mathbb{P}(\mathcal{E} \oplus \mathcal{O}_X)}(-1))^{r-i}),\]
where $p : \mathbb{P}(\mathcal{E} \oplus \mathcal{O}_X) \rightarrow X$ is the structure morphism. Let $s : X \longrightarrow \mathbb{P}(\mathcal{E} \oplus \mathcal{O}_X)$ be its section. The Thom isomorphism is given by 
\[th: \mathbf{E}^{*-2r,*-r}(X) \stackrel{\cong}{\longrightarrow} \mathbf{E}^{*,*}(Th_X(\mathbb{V}(\mathcal{E}))), \quad x \mapsto x \cup t_{\mathbb{V}(\mathcal{E})}.\]
Its inverse is the composition 
\[\xymatrix{\mathbf{E}^{*,*}(Th_X(\mathbb{V}(\mathcal{E})) \ar@{>->}[r]^{q^*} & \mathbf{E}^{*,*}(\mathbb{P}(\mathcal{E} \oplus \mathcal{O}_X)) \ar[r]^{\quad s^*} & \mathbf{E}^{*-2r,*-r}(X) \\ & \bigoplus_{i=0}^r \mathbf{E}^{*-2i,*-i}(X) \ar[u]^{\cong} }\]
which sends 
\[\sum_{i=0}^r p^*(x_i)\cup (-c_1(\mathcal{O}_{\mathbb{P}(\mathcal{E} \oplus \mathcal{O}_X)}(-1))^{r-i}) \mapsto x_r. \]
If $\xi$ is the universal quotient bundle on $\mathbb{P}(\mathcal{E} \oplus \mathcal{O}_X)$, i.e. there is an exact sequence 
\[0 \longrightarrow \mathcal{O}_{\mathbb{P}(\mathcal{E} \oplus \mathcal{O}_X)}(-1) \longrightarrow p^*(\mathcal{E} \oplus \mathcal{O}_X) \longrightarrow \xi \longrightarrow 0. \]
Then by Whitney sum formula we have 
\[q^*(t_{\mathbb{V}(\mathcal{E})}) = c_r(\xi).\]
Now let $KGL_S \in SH(S)$ be the motivic ring spectrum representing the algebraic $K$-theory constructed in \cite{Rio10}, \cite[\S 13]{CD12}. If $S$ is regular and $X \in Sm/S$ is a smooth $S$-scheme, then one has a natural isomorphism 
\begin{equation}\label{eq2}
\mathrm{Hom}_{SH(S)}(\Sigma^{\infty}X_+,KGL_S) \stackrel{\cong}{\longrightarrow} K_0(X).
\end{equation} 
$KGL$ is an oriented motivic ring spectrum. Indeed, by Bott periodicity one has an isomorphism 
\[KGL^{0,0}(S) \cong K_0(S) \cong \widetilde{KGL}^{0,0}(\mathbb{P}^1_S,\infty).\]
The Bott element $\beta$ is the image of $1$ under this isomorphism and $\beta \in KGL^{-2,-1}(S)$. The orientation of $KGL$ is given by 
\[c^{KGL} \stackrel{defn}{=} \beta^{-1} \cdot (1 - [\mathcal{O}_{\mathbb{P}_S^{\infty}}(1)]) \in KGL^{2,1}(\mathbb{P}^{\infty}_S).\]
For any line bundle $L \in Pic(S)$, its first Chern class is 
\[c^{KGL}_1(L) = \beta^{-1}(1-L^{\vee}).\] 
For an arbitrary morphism $f: T \rightarrow S$ of regular schemes one has a commutative diagram \cite[\S 13.1]{CD12}
\[\xymatrix{\mathrm{Hom}_{SH(S)}(\mathbf{1}_S,KGL_S) \ar[r]  \ar[d]_{\cong} & \mathrm{Hom}_{SH(T)}(f^*\mathbf{1}_S,f^*KGL_S) \ar@{=}[r] & \mathrm{Hom}_{SH(T)}(\mathbf{1}_T,KGL_T) \ar[d]^{\cong} \\ K_0(S) \ar[rr]_{f^*} && K_0(T)}\] 
Now assume $\mathcal{E}$ is a locally free sheaf of finite rank on a smooth scheme $X$ over a regular base $S$. We have a homotopy cofiber sequence (cf. \cite[\S 3]{MV99}) 
\[\mathbb{V}(\mathcal{E}) - s_0(X) \longrightarrow \mathbb{V}(\mathcal{E}) \longrightarrow Th_X(\mathbb{V}(\mathcal{E})),\]
where $Th_X(\mathbb{V}(\mathcal{E}))$ denotes the Thom space of the vector bundle $\mathbb{V}(\mathcal{E})$. This homotopy cofiber sequence gives rise to an exact sequence 
\begin{multline*}
\cdots \longrightarrow \mathrm{Hom}_{SH(S)}(\Sigma^{\infty}Th_X(\mathbb{V}(\mathcal{E})),KGL_S) \longrightarrow \mathrm{Hom}_{SH(S)}(\Sigma^{\infty} \mathbb{V}(\mathcal{E})_+,KGL_S) \\ \longrightarrow \mathrm{Hom}_{SH(S)}(\Sigma^{\infty}(\mathbb{V}(\mathcal{E})- s_0(X))_+,KGL_S) 
\end{multline*}
The natural isomorphism \ref{eq2} gives rise to a commutative diagram with exact rows (cf. \cite[\S 13.4, (K6a)]{CD12}), where we abbreviate $[-,-]$ for $\mathrm{Hom}_{SH(S)}(-,-)$: 
\[\xymatrix{[Th_X(\mathbb{V}(\mathcal{E})),KGL_S] \ar[r] & [\mathbb{V}(\mathcal{E}),KGL_S] \ar[d]_{\cong} \ar[r] & [\mathbb{V}(\mathcal{E})-s_0(X),KGL_S] \ar[d]^{\cong} \\  K_0(X) \ar[r] & K_0(\mathbb{V}(\mathcal{E})) \ar[r] & K_0(\mathbb{V}(\mathcal{E})-s_0(X)) }\]
The bottom row is the exact sequence \ref{eq1}. For a Noetherian scheme $X$ and an $\mathcal{O}_X$-module of finite type $\mathcal{E}$ the functor
\[\pi_* : Coh(\mathbb{V}(\mathcal{E})) \longrightarrow Coh(X)\]
is exact, since $\pi$ is an affine morphism, which implies that $R^i\pi_*\mathcal{F} = 0, \forall i > 0, \forall \mathcal{F} \in Coh(\mathbb{V}(\mathcal{E}))$. So we still can define the pushforward 
\[\pi_*: K_0(\mathbb{V}(\mathcal{E})) \longrightarrow K_0(X).\] 
Hence, one may try to define the topological index by applying the homomorphism 
\[K_0(Th_X(\mathbb{V}(\Omega^1_{X/k}))) \longrightarrow K_0(\mathbb{V}(\Omega^1_{X/k}))\] 
then composing with $f_*\pi_*$. However, $\pi: \mathbb{V}(\mathcal{E}) \rightarrow X$ is not necessary projective, so we may run into difficulties, when we apply later the Grothendieck-Riemann-Roch theorem. Now we restrict ourselves to the situation where $f : X \rightarrow \mathrm{Spec}(k)$ is a smooth proper scheme over a field $k$. $\Omega^1_{X/k}$ is a locally free $\mathcal{O}_X$-module of finite rank. Let $D: \mathcal{F} \rightarrow \mathcal{G}$ be an elliptic operator of order $\leq n$. Let $\pi : \mathbb{V}(\Omega^1_{X/k}) \longrightarrow X$ denotes the vector bundles associated to $\Omega^1_{X/k}$. The symbol $\sigma(D)$ defines then an element 
\[[\mathcal{P}^n_{\mathbb{V}(\Omega^1_{X/k})/X}]\cdot[\pi^*\mathcal{F}] - [\pi^*\mathcal{G}] \in K_0(\mathbb{V}(\Omega^1_{X/k})),\] 
which lies in the image of the homomorphism 
\[K_0(Th_X(\mathbb{V}(\Omega^1_{X/k}))) \longrightarrow K_0(\mathbb{V}(\Omega^1_{X/k})).\]
Let us denote by $[\sigma(D)] \in K_0(Th_X(\mathbb{V}(\Omega^1_{X/k})))$ a class, which is mapped to  
\[[\mathcal{P}^n_{\mathbb{V}(\Omega^1_{X/k})/X}]\cdot[\pi^*\mathcal{F}] - [\pi^*\mathcal{G}],\]
such that its image in $K_0(\mathbb{V}(\Omega^1_{X/k})-s_0(X))$ is trivial. We call this class $[\sigma(D)]$ a symbol class associated to the symbol $\sigma(D)$.  
\begin{defn}
Let $f : X \longrightarrow \mathrm{Spec}(k)$ be a smooth proper $k$-scheme. Let $D : \mathcal{F} \longrightarrow \mathcal{G}$ be an elliptic operator, where $\mathcal{F}, \mathcal{G} \in Coh(X)$. The topological index of $D$ is defined as 
\[ind_{top}(D) = f_*(th^{-1}([\sigma(D)])),\]
where 
\[th^{-1}: K_0(Th_X(\mathbb{V}(\Omega^1_{X/k}))) \stackrel{\cong}{\longrightarrow} K_0(X)\]
is the inverse Thom isomorphism of algebraic $K$-theory and 
\[f_* : K_0(X) \rightarrow K_0(\mathrm{Spec}(k)) = \mathbb{Z}, \quad \mathcal{E} \mapsto \sum_i(-1)^i[H^i_{Zar}(X,\mathcal{E})] = \chi(\mathcal{E}) \]
\end{defn}
If $S$ is any scheme, let $H\mathbb{Q}$ denote the Beilinson rational motivic cohomology ring spectrum in $SH(S)$ constructed in \cite{Rio10}, \cite[Defn. 14.1.2]{CD12}. Let 
\[ch_t: KGL_{\mathbb{Q}} \stackrel{\cong}{\longrightarrow} \vee_{i \in \mathbb{Z}}H\mathbb{Q} \wedge S^{2i,i} \]
be the total Chern character, which is an isomorphism of rational spectra, if $S = \mathrm{Spec}(k)$, where $k$ is a perfect field \cite[Defn. 6.2.3.9, Rem. 6.2.3.10]{Rio10}. One has the following result due to Riou: 
\begin{thm}\cite[Thm. 6.3.1]{Rio10}(Grothendieck-Riemann-Roch) \label{GRR} Let $k$ be a perfect field. Let $f: X \rightarrow S$ be a smooth projective morphism of smooth $k$-schemes. There is a commutative diagram in $SH(S)$
\[\xymatrix{\mathbf{R}f_{\star}KGL_{\mathbb{Q}} \ar[d]_{f_{\star}} \ar[rrrr]^{\mathbf{R}f_{\star}(ch\cdot Td(T_f))} &&&& \bigvee_{i \in \mathbb{Z}}\mathbf{R}f_{\star}H\mathbb{Q} \wedge S^{2i,i} \ar[d]^{f_{\star}} \\ KGL_{\mathbb{Q}} \ar[rrrr]_{ch_t} &&&& \bigvee_{i \in \mathbb{Z}}H\mathbb{Q} \wedge S^{2i,i} }\]
\end{thm} 
As explained in \cite[\S 13.7]{CD12} $f_{\star}$ induces the usual pushforward $f_*$ on $K$-theory. Now we can compute 
\begin{prop}
Let $f : X \rightarrow \mathrm{Spec}(k)$ be a smooth projective scheme over a perfect field $k$. Let $D: \mathcal{F} \rightarrow \mathcal{G}$ be an elliptic operator between coherent $\mathcal{O}_X$-modules. Then one has a formula 
\[ind_{top}(D) = \int_X ch(th^{-1}([\sigma(D)]))\cup Td(T_{X/k}),\]
where $\int_X$ means that we take the pushforward on motivic cohomology 
\[\int_X: H^{2*,*}_{\mathcal{M}}(X,\mathbb{Q}) \longrightarrow H_{\mathcal{M}}^{2*-2dim(X),*-dim(X)}(\mathrm{Spec}(k),\mathbb{Q}) = \mathbb{Q}.\]
\end{prop}
\begin{proof}
This is a trivial consequence of Thm. \ref{GRR}, where $S = \mathrm{Spec}(k)$. 
\end{proof}
Now let $D \in \mathrm{Diff}^n_{X/A}(\mathcal{F},\mathcal{G})$ be an arbitrary differential operator of order $\leq n$ on a projective scheme $X \rightarrow \mathrm{Spec}(A)$ over a Noetherian ring $A$ and $\mathcal{F}, \mathcal{G} \in Coh(X)$. As $D$ gives rise to an $\mathcal{O}_X$-module homomorphism 
\[\mathcal{P}^n_{X/A} \otimes_{\mathcal{O}_X} \mathcal{F} \longrightarrow \mathcal{G},\]
we may take the twisting 
\[\mathcal{P}^n_{X/A} \otimes_{\mathcal{O}_X} \mathcal{F} \otimes_{\mathcal{O}_X} \mathcal{O}_X(m) \longrightarrow \mathcal{G} \otimes_{\mathcal{O}_X} \mathcal{O}_X(m),\]
which in turn gives us a differential operator of order $\leq n$
\[D(m) : \mathcal{F}(m) \longrightarrow \mathcal{G}(m).\]
We call $D(m)$ a twisting of $D$.
\begin{prop}
Let $X \longrightarrow \mathrm{Spec}(k)$ be a smooth projective scheme over a perfect field $k$ and $D : \mathcal{F} \longrightarrow \mathcal{G}$ be a differential operator of order $\leq n$ between coherent $\mathcal{O}_X$-modules. Then there exists a number $N_0$, such that for all $N \geq N_0$ one has
\[Ind(D(N)) = \int_X ch(([\mathcal{P}^n_{X/k}]\cdot [\mathcal{F}]-[\mathcal{G}])(N)) \cup Td(T_{X/k}).\]
\end{prop}   
\begin{proof}
This is quite trivial. By the result of Serre (see e.g \cite[Thm. 2.2.1, Prop. 2.2.2]{EGA3}), there is a number $n_0$, such that for all $q > 0$ and all $a \geq n_0$
\[H^q_{Zar}(X,\mathcal{P}^n_{X/k} \otimes_{\mathcal{O}_X} \mathcal{F}(a)) = 0,\]
and a number $m_0$, such that for all $q > 0$ and all $b \geq m_0$ 
\[H^q_{Zar}(X,\mathcal{G}(b)) = 0.\]
We take $N_0 = max(n_0,m_0)$ to be the maximal of $n_0$ and $m_0$. We have trivially that for all $N \geq N_0$  
\[Ind(D(N)) = dim_k \Gamma_{mod}(X,\mathcal{P}^n_{X/k}\otimes_{\mathcal{O}_X} \mathcal{F}(N)) - dim_k \Gamma_{mod}(X,\mathcal{G}(N)).\] 
The proposition follows now easily from the Hirzebruch-Riemann-Roch theorem 
\[dim_k \Gamma_{mod}(X,E(N)) = \int_X ch(E(N)) \cup Td(T_{X/k}).\]
\end{proof}
Now we are ready to state the following algebraic weak form of the Atiyah-Singer index theorem 
\begin{thm}\label{mainthm}
Let $X \longrightarrow \mathrm{Spec}(k)$ be a smooth projective scheme over a perfect field $k$ and $D \in \mathrm{Diff}^n_{X/k}(\mathcal{F},\mathcal{G})$ be an elliptic differential operator of order $\leq n$, where $\mathcal{F},\mathcal{G} \in Coh(X)$ are coherent $\mathcal{O}_X$-modules. There exists a number $N_0$, such that for all $N \geq N_0$ there is an equality 
\[Ind(D(N)) = ind_{top}(D(N)).\]
\end{thm}
\begin{proof}
It remains to prove that 
\[th^{-1}([\sigma(D(N))]) = ([\mathcal{P}^n_{X/k}]\cdot [\mathcal{F}]-[\mathcal{G}])(N),\]
which means that we have to show there is a commutative diagram 
\[\xymatrix{K_0(Th_X(\mathbb{V}(\Omega^1_{X/k}))) \ar[dr]_{th^{-1}} \ar[r] & K_0(\mathbb{V}(\Omega^1_{X/k})) \ar[d]^{s_0^*} \\ & K_0(X) }\]
By construction, we have a commutative diagram 
\[\xymatrix{K_0(Th_X(\mathbb{V}(\Omega^1_{X/k}))) \ar[rd]_{th^{-1}} \ar[r] & K_0(\mathbb{P}(\Omega^1_{X/k} \oplus \mathcal{O}_X)) \ar[d]^{s^*} \\ & K_0(X) }\]
So it remains to see that we have a commutative diagram 
\[\xymatrix{K_0(\mathbb{V}(\Omega^1_{X/k})) \ar[d]_{s_0^*} \ar[r] & K_0(\mathbb{P}(\Omega^1_{X/k} \oplus \mathcal{O}_X)) \ar[d]^{s^*} \\ K_0(X) \ar@{=}[r] & K_0(X)}\]
But this is quite obvious, since the closed immersion $s: X \longrightarrow \mathbb{P}(\Omega^1_{X/k} \oplus \mathcal{O}_X)$ is the composition of the $0$-section $s_0: X \longrightarrow \mathbb{V}(\Omega^1_{X/k})$ and $\mathbb{V}(\Omega^1_{X/k}) \longrightarrow \mathbb{P}(\Omega^1_{X/k} \oplus \mathcal{O}_X)$ (cf. \cite[Prop. 8.3.2]{EGA2}).
\end{proof}
Finally, we will give in the last part of this paper a few small examples. For this purpose, we always restrict now ourselves to the case of a smooth projective $k$-scheme $X \longrightarrow \mathrm{Spec}(k)$ of dimension $\mathrm{dim}(X) = d$, where $k$ is a perfect field. Let us begin with the following lemma:   
\begin{lem}(de Jong and Starr \cite[Lem. 2.1]{dJS06})\label{lem1}
Let $X \subset \mathbb{P}^n$ be a smooth complete intersection of type $(d_1,\cdots,d_c)$. Then 
\[ch(T_{X}) = n - c + \sum_{i=1}^{n-c}(n+1-\sum_{i=1}^cd_i^k)\frac{c_1(\mathcal{O}_X(1))^k}{k!}.\]
\end{lem}  
\begin{proof}
The proof is so elementary, so we reproduce the proof. Let $i: X \longrightarrow \mathbb{P}^n$ denote the closed embedding. Consider the short exact sequence 
\[0 \longrightarrow T_X \longrightarrow i^*T_{\mathbb{P}^n} \longrightarrow \bigoplus_{i=1}^c \mathcal{O}_X(d_i) \longrightarrow 0.\]
Therefore, 
\[ch(T_X) = ch(i^*T_{\mathbb{P}^n}) - c - \sum_{k=1}^{n-c}\sum_{i=1}^c d_i^k\frac{c_1(\mathcal{O}_X(1))^k}{k!}.\] 
Consider the Euler sequence 
\[0 \longrightarrow \mathcal{O}_{\mathbb{P}^n} \longrightarrow \mathcal{O}_{\mathbb{P}^n}(1)^{\oplus (n+1)} \longrightarrow T_{\mathbb{P}^n} \longrightarrow 0.\]
One has 
\[ch(T_{\mathbb{P}^n}) = n + \sum_{k=1}^n(n+1)\frac{c_1(\mathcal{O}_{\mathbb{P}^n}(1))^k}{k!}.\]
The lemma follows easily. 
\end{proof}
Consider the Atiyah class in $\mathrm{Ext}^1_{\mathcal{O}_X}(\mathcal{O}_X,\Omega^1_{X/k})$ 
\[0 \longrightarrow \Omega^1_{X/k} \longrightarrow \mathcal{P}^1_{X/k} \stackrel{u}{\longrightarrow} \mathcal{O}_X \longrightarrow 0.\]
The morphism $u : \mathcal{P}^1_{X/k} \longrightarrow \mathcal{O}_X$ defines a differential operator of order $\leq 1$
\[D_u \in \mathrm{Diff}^1_{X/k}(\mathcal{O}_X,\mathcal{O}_X).\] 
\begin{lem}
$D_u$ is an elliptic operator. 
\end{lem}
\begin{proof}
Let $\pi : \mathbb{V}(\Omega^1_{X/k}) \longrightarrow X$ be the vector bundle associated to $\Omega^1_{X/k}$. We pullback the Atiyah class via $\pi$ to obtain an exact sequence
\[0 \longrightarrow \Omega^1_{\mathbb{V}(\Omega^1_{X/k})/X} \longrightarrow \mathcal{P}^1_{\mathbb{V}(\Omega^1_{X/k})/X} \longrightarrow \mathcal{O}_{\mathbb{V}(\Omega^1_{X/k})} \longrightarrow 0.\]
For each point $P \in \mathbb{V}(\Omega^1_{X/k})$ we have by \cite[Cor. 16.4.12]{EGA4} 
\[\mathcal{P}^1_{\mathbb{V}(\Omega^1_{X/k})/X, P} = \mathcal{O}_{\mathbb{V}(\Omega^1_{X/k}),P}/\mathfrak{m}_P^2.\]
The morphism $\pi^*(u) : \mathcal{P}^1_{\mathbb{V}(\Omega^1_{X/k})/X} \longrightarrow \mathcal{O}_{\mathbb{V}(\Omega^1_{X/k})}$ becomes obviously an isomorphism at each point $P \in \mathbb{V}(\Omega^1_{X/k}) - s_0(X)$. 
\end{proof}
From now on we always consider the elliptic differential operator $D_u$. We write $H = c_1(\mathcal{O}_X(1))$ and for a curve $X$ 
\[deg(H) = \int_X c_1(\mathcal{O}_X(1)).\]
\begin{prop}
Let $C \subset \mathbb{P}^3$ be a smooth curve of genus $1$, which is a smooth complete intersection of two quadrics. For all integer number $n \geq 0$ one has
\[ch(\Omega^1_{C/k}(n))\cup Td(T_{C/k}) \in H^{2*,*}_{\mathcal{M}}(C,\mathbb{Z}).\]
For $N >> 0$ one has 
\[Ind(D_u(N)) =  4N.\]
\end{prop}
\begin{proof}
$C$ is smooth complete intersection of two quadrics. By the lemma \ref{lem1}  we have $ch(\Omega^1_{C/k}) = 1$. This implies 
\[ch(\Omega^1_{C/k}(n)) = ch(\mathcal{O}_C(n)) = (1+n\cdot H),\]
for any number $n \geq 0$, where we write $H = c_1(\mathcal{O}_C(1))$ for the hyperplane section. As $g_C = 1$, the canonical divisor $K_C = 0$ is trivial. So we have $Td(T_{C/k}) = 1$ and hence 
\[ch(\Omega^1_{C/k}(n))\cup Td(T_{C/k}) = (1+n\cdot H) \in H^{2*,*}_{\mathcal{M}}(C,\mathbb{Z}).\] 
This implies easily that $Ind(D_u(N)) = N \cdot deg(H) = 4N$ for $N >> 0$.
\end{proof}
\begin{prop}
Let $X \longrightarrow \mathbb{P}^1$ be a flat projective morphism over a perfect field $K$, such that the generic fiber $X_{\eta}$ is a smooth curve of genus one in $\mathbb{P}^2$. Then there exists a number $N >> 0$, such that 
\[Ind(D_u) \equiv 1 \mod 2.\]
\end{prop}
\begin{proof}
We have for $N >> 0$:  
\[Ind(D_u(N)) = dim_K \Gamma_{mod}(X,\Omega^1_{X/K}(N)).\]
We take two embeddings $i_1 : \mathbb{P}^1 \longrightarrow \mathbb{P}^M$ and $i_2 : X \longrightarrow \mathbb{P}^M$, such that they are compatible with $f$
\[\xymatrix{X \ar[rr]^f \ar[rrd]_{i_2} && \mathbb{P}^1 \ar[d]^{i_1} \\ && \mathbb{P}^M }\]
Let $N_0 >> 0$ be a big integer number. By \cite[Cor. 7.9.13, Seconde partie]{EGA3} $f_*(\Omega^1_{X/K}(N_0))$ is a locally free $\mathcal{O}_{\mathbb{P}^1}$-module of rank $\chi(X_{\eta},\Omega^1_{X_{\eta},\eta}(N_0))$. By flat base change for Zariski cohomology of coherent sheaves it is enough to compute $\chi(X_{\bar{\eta}},\Omega^1_{X_{\bar{\eta}}}(N_0))$, where $X_{\bar{\eta}}$ is the base change of $X_{\eta}$ to an algebraic closure $\bar{\eta}$. $X_{\bar{\eta}}$ must be an elliptic curve. For an elliptic curve $(E,P_0)$, we know that $\Omega^1_E \cong \mathcal{O}_E$ and $\mathcal{L}(3P_0)$ is very ample, i.e. $\mathcal{L}(3P_0) \cong \mathcal{O}_E(1)$. The last isomorphism means simply that we can embed 
\[|3P_0| : E \longrightarrow \mathbb{P}^2.\]
By Riemann-Roch theorem for curves we also have $dim \, H^0(E,\mathcal{L}(nP_0)) = n$. Now we can apply the theorem of Grothendieck for the decomposition of vector bundles on $\mathbb{P}^1$ and we have
\[f_*(\Omega^1_{X/K}(N_0)) \cong \bigoplus_{i=1}^{3N_0}\mathcal{O}_{\mathbb{P}^1}(a_i), \quad a_i \in \mathbb{Z}. \]
Let $N_0' >> 0$ be another big integer number. We have
\[f_*(\Omega^1_{X/K}(N_0))(N_0') \cong \bigoplus_{i=1}^{3N_0}\mathcal{O}_{\mathbb{P}^1}(a_i+N_0').\]
So 
\[dim_K f_*(\Omega^1_{X/K}(N_0))(N_0') = \sum_{i=1}^{3N_0}(1+a_i+N_0') = 3N_0(N_0'+1) + \sum_{i=0}^{3N_0}a_i.\]
If this dimension is already odd, then there is nothing to prove. So we assume 
\[3N_0(N_0'+1) + \sum_{i=1}^{3N_0}a_i \equiv 0 \mod 2.\]
Let $N_0'' >> 0$ be another big integer number. After twisting by $\mathcal{O}(N_0'')$ we have 
\[dim_K f_*(\Omega^1_{X/K}(N_0))(N_0'+N_0'') = 3N_0(N_0'+N_0''+1) + \sum_{i=0}^{3N_0}a_i.\]
Now we take simply $N_0 \equiv 1 \mod 2$, $N_0' \equiv 1 \mod 2$ and $N_0'' \equiv 0 \mod 2$. We obtain 
\[dim_Kf_*(\Omega^1_{X/K}(N_0))(N_0'+N_0'') \equiv 1 \mod 2\]
By projection formula we have for any number $n > 0$ an isomorphism 
\begin{multline*}
f_*(\Omega^1_{X/K}) \otimes_{\mathcal{O}_{\mathbb{P}^1}} \mathcal{O}_{\mathbb{P}^1}(n) \stackrel{\cong}{\longrightarrow} f_*(\Omega^1_{X/K}) \otimes_{\mathcal{O}_{\mathbb{P}^1}}i_1^*\mathcal{O}_{\mathbb{P}^M}(n) \stackrel{\cong}{\longrightarrow} f_*(\Omega^1_{X/K} \otimes_{\mathcal{O}_X} f^*i_1^*\mathcal{O}_{\mathbb{P}^M}(n)) \\ \stackrel{\cong}{\longrightarrow} f_*(\Omega^1_{X/K} \otimes_{\mathcal{O}_X} i_2^*\mathcal{O}_{\mathbb{P}^M}(n)) \stackrel{\cong}{\longrightarrow} f_*(\Omega^1_{X/K} \otimes_{\mathcal{O}_X} \mathcal{O}_X(n)).
\end{multline*}
This implies 
\[
dim_K \Gamma_{mod}(X,\Omega^1_{X/K}(N_0+N_0'+N_0'')) = dim_K \Gamma_{mod}(\mathbb{P}^1,f_*(\Omega^1_{X/K}(N_0+N_0'+N_0''))) \equiv 1 \mod 2, \]
where the first equality follows easily from the adjunction 
\[f^* : \mathcal{O}_{\mathbb{P}^1}-Mod \leftrightarrows \mathcal{O}_X-Mod : f_*.\]
Therefore, we can conclude that there exists a number $N = N_0+N_0' + N_0'' >> 0$ such that $Ind(D_u(N)) \equiv 1 \mod 2$, which finishes the proof.       
\end{proof}
\begin{prop}
Let $X \subset \mathbb{P}^4$ be a smooth $K3$-surface, which is a smooth complete intersection, over a perfect field $k$. 
\begin{enumerate}
\item If $deg(X) = 4$, then for a big number $N >> 0$ one has 
\[Ind(D_u(N)) = 4N^2 - 36N +1.\]
\item If $deg(X) = 6$,  then for a big number $N >> 0$ one has
\[Ind(D_u(N)) = 6N^2 - 60N - 20.\]
\end{enumerate} 
\end{prop}
\begin{proof}
We write $c_1(\mathcal{O}_X(1)) = H$ and $K$ for the canonical divisor. For an algebraic surface one has 
\[Td(T_{X/k}) = 1 - \frac{1}{2}K + \frac{1}{12}(K^2+c_2).\]
For a smooth projective $K3$-surface one has $K = 0$ and $K^2 = 0$. $\Omega^1_{X/k}$ is the dual bundle of $T_{X/k}$. So we have by lemma \ref{lem1} 
\[ch(\Omega^1_{X/k}) = 2 - (5+d_1+d_2)\cdot H + (5-d_1^2-d_2^2) \cdot \frac{H^2}{2}.\]
For any number $n \geq 0$ one has
\[ch(\mathcal{O}_X(n)) = 1 + n\cdot H + \frac{1}{2}n^2\cdot H^2. \]
If $deg(X) = 4$ then $(d_1,d_2) = (2,2)$ and if $deg(X)=6$ then $(d_1,d_2) = (2,3)$. The result follows now easily,  since $deg(X) = \int_X H^2$ and $\int_Xc_2 = 24$. 
\end{proof}
\bibliographystyle{plain}
\renewcommand\refname{References}

\end{document}